\documentclass[12pt]{amsart}

\usepackage[latin1]{inputenc}
\usepackage{amsmath}
\usepackage{amsfonts}
\usepackage{amssymb}
\usepackage{graphics}
\usepackage{amssymb,amsmath,amsthm,amscd,epsf,latexsym,verbatim,graphicx,amsfonts}
\input epsf.tex

\newtheorem{theorem}{Theorem}[section]

\newtheorem{lemma}[theorem]{Lemma}
\newtheorem{corollary}[theorem]{Corollary}
\theoremstyle{definition}

\numberwithin{equation}{section}

\newcommand{\bbZ}{\mathbb{Z}}
\newcommand{\bbR}{\mathbb{R}}
\newcommand{\bbC}{\mathbb{C}}

\newcommand{\bbD}{\mathbb{D}}
\newcommand{\bbH}{\mathbb{H}}
\newcommand{\bbN}{\mathbb{N}}

\newcommand{\dtpi}{\frac{d\theta}{2\pi}}

\newcommand{\eitheta}{e^{i\theta}}

\newcommand{\mcm}{\mathcal{M}}

\newcommand{\mcn}{\mathcal{N}}
\newcommand{\Real}{\textrm{Re}}

\newcommand{\bard}{\overline{\bbD}}
\newcommand{\barc}{\overline{\bbC}}
\newcommand{\barg}{\overline{G}}
\newcommand{\supp}{\textrm{supp}}
\newcommand{\pch}{\textrm{Pch}}
\newcommand{\ch}{\textrm{ch}}
\newcommand{\nch}{\not\in\textrm{ch}}
\newcommand{\ltm}{L^2(\mu)}

\newcommand{\nri}{n\rightarrow\infty}

\begin{document}

\title[ ] {A New Approach to Ratio Asymptotics for Orthogonal Polynomials}

\bibliographystyle{plain}

\thanks{  }

\vspace{-7mm}

\maketitle

\begin{center}
\textbf{Brian Simanek}\\
\vspace{2mm}
\small{Mathematics MC 253-37, California Institute of Technology}\\
\small{Pasadena, CA 91125, USA. E-mail: bsimanek@caltech.edu}
\end{center}

\begin{abstract}
We use a non-linear characterization of orthonormal polynomials due to Saff in order to show that the behavior of orthonormal polynomials is uniquely determined by its normalization and its leading coefficient.  Several applications of this result are also discussed.  One of our main theorems is that for regular measures on the closed unit disk - including, but not limited to the unit circle - one has ratio asymptotics along a sequence of asymptotic density $1$.
\end{abstract}

\vspace{4mm}

\footnotesize\noindent\textbf{Keywords:} Orthogonal polynomials, Ratio asymptotics, Relative Asymptotics, Regular measures

\vspace{2mm}

\noindent\textbf{Mathematics Subject Classification:} 42C05

\vspace{2mm}

\normalsize

\section{Introduction}\label{intro}

Throughout this paper, we will consider a compactly supported, positive, and finite measure $\mu$ with infinite support in the complex plane.  Given such a measure, one can perform Gram-Schmidt orthogonalization on the sequence $\{1,z,z^2,z^3,\ldots\}$ in the space $L^2(\mu)$ and arrive at the sequence $\{p_n(z;\mu)\}_{n=0}^{\infty}$ of orthonormal polynomials, which satisfy
\[
\int_{\bbC}\overline{p_n(z;\mu)}p_m(z;\mu)d\mu(z)=\delta_{nm}.
\]
We will write $p_n(z;\mu)=\kappa_nz^n+\cdots$, where $\kappa_n>0$.  The polynomial $p_n(z;\mu)\kappa_n^{-1}$ is a monic polynomial and will be denoted by $\Phi_n(z;\mu)$.  This polynomial satisfies
\[
\|\Phi_n(\cdot;\mu)\|_{L^2(\mu)}=\inf\left\{\|Q\|_{L^2(\mu)}:Q=z^n+\mbox{lower order terms}\right\},
\]
a property called the \textit{extremal property}.  We also note that $\kappa_n^{-1}=\|\Phi_n(\cdot;\mu)\|_{\ltm}$.

For a measure $\mu$ with compact support $\supp(\mu)$, we will let $\ch(\mu)$ denote the convex hull of the support of $\mu$ and $\pch(\mu)$ denote the polynomial convex hull of the support of $\mu$, where the polynomial convex hull of a set $X$ is defined as in \cite{StaTo} by
\[
\pch(X)=\bigcap_{\textrm{polynomials $p\neq0$}}\left\{z:|p(z)|\leq\|p\|_{L^{\infty}(X)}\right\}.
\]
It is not difficult to see that if $\Omega$ is the unbounded component of $\barc\setminus\supp(\mu)$ then $\pch(\mu)=\barc\setminus\Omega$ (see \cite{StaTo}).

One is often interested in the behavior of the polynomials $p_n$ in various regions of the plane.  The most general results concern the behavior of $p_n$ in $\barc\setminus\pch(\mu)$, where one is often able to employ potential theoretic techniques (see for example \cite{StaTo,Wid2}).  Our main goal will be to show that the behavior of the polynomial $p_n(z;\mu)$ when $|z|$ is sufficiently large is determined only by its leading coefficient and the fact that it has $L^2(\mu)$-norm equal to $1$.  More precisely, in Theorem \ref{nis} below, we will show that any other polynomial of the same degree having approximately the same leading coefficient and approximately the same $L^2(\mu)$-norm has the same behavior when $|z|$ is sufficiently large.

The primary tool will be a non-linear characterization of the polynomials $\{p_n(z;\mu)\}_{n\geq0}$ originally proven by Saff in \cite{SaffConj}, whose proof proceeds as follows.  The orthogonality relation implies that if $\mbox{deg}(Q)\leq n$ then
\[
\int_{\bbC}\overline{p_n(w;\mu)}\frac{Q(z)-Q(w)}{z-w}d\mu(w)=0.
\]
This of course immediately shows that for $z\not\in\supp(\mu)$ one has
\begin{align}\label{interim}
Q(z)\int_{\bbC}\frac{\overline{p_n(w;\mu)}}{z-w}d\mu(w)=
\int_{\bbC}\frac{\overline{p_n(w;\mu)}Q(w)}{z-w}d\mu(w).
\end{align}
Setting $Q=p_n(\cdot;\mu)$ and dividing shows that for any polynomial $Q$ of degree at most $n$, we have
\begin{align}\label{king}
\frac{Q(z)}{p_n(z;\mu)}=\frac{\int_{\bbC}\frac{\overline{p_n(w;\mu)}Q(w)}{z-w}d\mu(w)}{\int_{\bbC}\frac{|p_n(w;\mu)|^2}{z-w}d\mu(w)},
\end{align}
whenever both denominators in (\ref{king}) are non-zero.

At first glance, the utility of (\ref{king}) is not obvious, though some applications are discussed in \cite{SaffConj}.  We will apply this formula in cases where $Q(z)=Q_n(z)$ is also $n$-dependent.  The key to our calculations will be to write the numerator on the right hand side of (\ref{king}) as a perturbation of the denominator and - under suitable hypotheses - show that the perturbation tends to zero as $\nri$ while the denominator does not.  In order to do so, we will require that the left hand side of (\ref{king}) tends to $1$ at infinity as $\nri$ and also that $Q_n(z)$ has $L^2(\mu)$-norm tending to $1$ as $\nri$.  Obviously $Q_n(z)=p_n(z;\mu)$ satisfies these conditions, but we will show that for any sequence $\{Q_n\}_{n\in\bbN}$ of polynomials satisfying these conditions, the left hand side of (\ref{king}) tends to $1$ as $\nri$ when $|z|$ is sufficiently large.

After proving the key result in the next section, we will apply it in the remaining sections.  We will apply it to prove several results, including a simplification of the conjecture in \cite{SaffConj} and some stability results for $p_n(z;\mu)$ under perturbations of the measure $\mu$.  Our most important results concern the ratio asymptotics of the orthonormal polynomials when the measure $\mu$ is \textit{regular} (see below) on the closed unit disk or a lemniscate.

One usually studies the asymptotics of $p_n(z;\mu)$ outside of $\ch(\mu)$ in one of three ways; the first is \textit{root asymptotics}:
\[
\lim_{\nri}|p_n(z;\mu)|^{1/n};
\]
the second is \textit{ratio asymptotics}:
\[
\lim_{\nri}\frac{p_n(z;\mu)}{p_{n-1}(z;\mu)};
\]
and the third is \textit{Szeg\H{o} asymptotics}:
\[
\lim_{\nri}\frac{p_n(z;\mu)}{\varphi(z)^n}\qquad,\qquad \varphi(z) \, \mbox{ analytic on } \, \bbC\setminus\ch(\mu).
\]
It is easy to see that the existence of the limit for Szeg\H{o} asymptotics implies the existence of the limit for ratio asymptotics, which in turn implies the existence of the limit for root asymptotics, and in general none of the converse statements hold.

A measure $\mu$ is called \textit{regular} if
\[
\lim_{\nri}\kappa_n^{-1/n}=\mbox{capacity}(\supp(\mu))
\]
(see \cite{Ransford} for an elementary discussion of capacity).  If $K$ is a compact set then a measure $\mu$ is said to be \textit{regular on $K$} if $\mu$ is regular, $\supp(\mu)\subseteq K$, and the boundary of  $\pch(K)$ is contained in $\supp(\mu)$.  Regularity is a necessary and sufficient condition for the existence of root asymptotics (see Theorem 3.1.1 in \cite{StaTo}).  Although ratio asymptotics need not hold for regular measures (see the example in section \ref{diskrat}), we can say something about the asymptotic behavior of $p_{n}/p_{n-1}$.  We prove that if the measure $\mu$ is regular on $\bard=\{z:|z|\leq1\}$, then the ratio $zp_{n-1}(z;\mu)/p_{n}(z;\mu)$ converges to $1$ uniformly on compact subsets of $\{z:|z|>1\}$ as $n$ tends to infinity through a subsequence of asymptotic density $1$.  Using the methods of \cite{RealRat}, one can show a similar result holds for compactly supported measures on $\bbR$ with essential support equal to $[-2,2]$.  We also show that if the measure $\mu$ is regular on the lemniscate $E_m=\{z:|z^m-1|\leq1\}$ then the ratio $(z^m-1)p_{n-m}(z;\mu)/p_{n}(z;\mu)$ converges to $1$ uniformly on compact subsets of $\barc\setminus\ch(\mu)$ as $n$ tends to infinity through a subsequence of asymptotic density $1$.  The advantage of working on a lemniscate is that there is a monic polynomial whose $L^{\infty}$-norm is $1$ on the lemniscate, while for more general supports this is not necessarily the case.  If this is not the case, then we cannot obtain convergence of $p_n/p_{n-1}$ by our methods, but we can describe the behavior of $p_n/p_{n-k_n}$ for a possibly unbounded sequence $\{k_n\}$ (see Section \ref{conn} for details).

The strength of our results is rooted in the weak assumptions we place on the measure $\mu$ in order to arrive at a ratio asymptotic result.  Many ratio asymptotic results arise as a consequence of Szeg\H{o} asymptotics (see \cite{SimBlob,Suetin}), which is a stronger conclusion than ratio asymptotics and hence requires stronger hypotheses on the measure.  In \cite{SaffConj}, Saff places bounds on $|p_n/p_{n-1}|$ for arbitrary compactly supported measures using methods similar to ours.  The results in \cite{Duran} concern orthogonal polynomials on the real line and are in the same spirit as our Theorem \ref{nis}, though Theorem \ref{nis} is much more general.

In addition to studying ratio asymptotics for consecutive orthonormal polynomials, we will also consider ratios of orthonormal polynomials corresponding to different but related measures.  In particular, we will study the \textit{Uvarov transform}, which is obtained by adding a point mass to the measure $\mu$:
\[
\mu_x=\mu+t\delta_x,\qquad t>0\,;
\]
and the \textit{Christoffel Transform}, which is obtained by multiplying $\mu$ by the square modulus of a monomial:
\[
d\nu^x(z)=|z-x|^2d\mu(z).
\]
In both cases, we show that the asymptotic behavior of the orthonormal polynomials outside of $\ch(\mu)$ is unchanged provided the pure point (for the Uvarov Transform) or the zero of the monomial (for the Christoffel Transform) satisfies the condition
\begin{align}\label{nevlike}
\lim_{\nri}\frac{|p_n(x;\mu)|^2}{\sum_{j=0}^{n-1}|p_j(x;\mu)|^2}=0
\end{align}
(see (\ref{nevailike}) in Section \ref{stabl}).  The condition of regularity is equivalent to
\[
\limsup_{\nri}|p_n(z;\mu)|^{1/n}=1
\]
for every $z$ in the outer boundary of the support of $\mu$, except perhaps on a set of capacity $0$ (see Theorem 3.1.1 in \cite{StaTo}).  Therefore, condition (\ref{nevlike}) - when applied to a point $x$ in the outer boundary of $\supp(\mu)$ - qualitatively tells us that $x$ is not a point at which $|p_n(x;\mu)|$ grows exponentially (see also Theorem 1.3 in \cite{BLS}).

After proving the key fact about ratios of polynomials in the next section, we will apply it in the case when the orthonormal polynomials correspond to a measure supported on the closed unit disk in Section \ref{dsk}.  We also include a brief digression where we show that if $\mu$ is any regular measure on $\bard$ then there is a subsequence $\mcn\subseteq\bbN$ of asymptotic density $1$ so that the probability measures $\{|p_n(z;\mu)|^2d\mu(z)\}_{n\geq0}$ converge weakly to normalized arc-length measure on the unit circle as $\nri$ through $\mcn$.  In Section \ref{conn}, we will apply the results of Section \ref{key} to orthonormal polynomials whose measure of orthogonality has a more general support.  The main theorem in Section \ref{conn} is analogous to results in Section \ref{dsk}, but requires a small sacrifice in the strength of the conclusion due to the added generality.  Finally, in Section \ref{stabl}, we will apply the results of Section \ref{key} to prove our stability results concerning orthonormal polynomials when the measure is perturbed in specific ways.  The foundation for all that follows is Theorem \ref{nis} in the next section.

\vspace{2mm}

\noindent\textbf{Acknowledgements.}  It is a pleasure to thank my advisor Barry Simon for his guidance during this research and Vilmos Totik for useful feedback on an early draft of this work.  Many thanks are also due to the anonymous referees for many suggestions about the content and exposition of this paper.  This material is based upon work supported by the National Science Foundation Graduate Research Fellowship under Grant No. DGE-1144469.

\section{The Key Fact}\label{key}

In this section, we will prove the crucial property mentioned in the introduction, which we will apply in later sections.  Before we prove our main result of this section, we make the following simple calculation:

\begin{lemma}\label{below}
Let $\mu$ be a measure with compact support $\supp(\mu)\subseteq\bbC$ and suppose $z$ satisfies $z\nch(\mu)$.  There is a constant $A_z>0$ so that
\[
\left|\int_{\bbC}\frac{|p_n(w;\mu)|^2}{z-w}d\mu(w)\right|>A_z
\]
for every $n\in\bbN$.  Furthermore, the constant $A_z$ may be bounded uniformly from below on any compact subset of $\bbC\setminus\ch(\mu)$.
\end{lemma}

\begin{proof}
Since $z\nch(\mu)$, we can find a $\theta\in\bbR$ so that $\min_{w\in\ch(\mu)}\Real[\eitheta z-\eitheta w]=\textrm{dist}(z,\ch(\mu))$.  Therefore
\begin{align*}
\Real\left[e^{-i\theta}\int_{\bbC}\frac{|p_n(w;\mu)|^2}{z-w}d\mu(w)\right]&=
\int_{\bbC}\frac{|p_n(w;\mu)|^2}{|z-w|^2}\Real[e^{-i\theta}\bar{z}-e^{-i\theta}\overline{w}]d\mu(w)\\
&\geq
\frac{\textrm{dist}(z,\ch(\mu))}{\sup_{w\in\supp(\mu)}|z-w|^2}
\end{align*}
as desired.  The uniformity in $A_z$ is now obvious.
\end{proof}

Lemma \ref{below} assures us that the integral on the left hand side of (\ref{interim}) is non-zero when $z\nch(\mu)$, for if it were zero then the right hand side of (\ref{interim}) would also vanish for \textit{every} choice of $Q$ and we have just shown that it cannot vanish for $Q=p_n(z;\mu)$.  Therefore, the division step in the derivation of (\ref{king}) is justified whenever $z\not\in\ch(\mu)$.  Furthermore, we have demonstrated that the denominator on the right hand side in (\ref{king}) is non-zero for appropriate $z$.

The following theorem will be used heavily for the applications in the remainder of this paper.  It tells us that the behavior of the orthonormal polynomials when $|z|$ is large is determined only by its normalization and its leading coefficient.

\begin{theorem}\label{nis}
Suppose $\mu$ is a (finite) and compactly supported measure on $\bbC$.  For each $n\in\bbN$, choose a polynomial $Q_n$ of degree exactly $n$ and leading coefficient $\tau_n$ so that the following properties hold:
\begin{enumerate}
\item\label{norm} $\lim_{\nri}\|Q_n\|_{L^2(\mu)}=1$,
\item\label{leading} $\lim_{\nri}\tau_n/\kappa_n=1$.
\end{enumerate}
Then
\begin{align}\label{keyconc}
\lim_{\nri}\frac{Q_n(z)}{p_n(z;\mu)}=1
\end{align}
for all $z\nch(\mu)$.  Furthermore, the convergence is uniform on compact subsets of $\barc\setminus\ch(\mu)$.
\end{theorem}

\vspace{2mm}

\noindent\textit{Remark 1.}  The proof will show that we get the same conclusion if we only define $Q_n$ for $n$ in some subsequence and then send $\nri$ through that subsequence.

\vspace{2mm}

\noindent\textit{Remark 2.}  By evaluating $Q_n(\cdot)/p_n(\cdot;\mu)$ at infinity, we see that the second condition in Theorem \ref{nis} is necessary for (\ref{keyconc}) to hold.  Additionally, since $\kappa_n^{-1}=\|\Phi_n(\cdot;\mu)\|_{\ltm}$, the second condition and the extremal property imply
\[
\liminf_{\nri}\|Q_n\|_{\ltm}\geq\liminf_{\nri}\tau_n\|\Phi_n(\cdot;\mu)\|_{\ltm}=1,
\]
so the first condition of Theorem \ref{nis} is really a statement about the $\limsup$.

\vspace{2mm}

\noindent\textit{Remark 3.}  We will show by means of an example in Section \ref{chr} that we cannot extend the conclusion of Theorem \ref{nis} to include the boundary of $\pch(\mu)$.  However, we will be able to say something about what happens at points $z$ that are outside $\pch(\mu)$, but inside the convex hull of the support of $\mu$ (see the end of Section \ref{dsk}).

\vspace{2mm}

\begin{proof}
Fix $z\nch(\mu)$.
By (\ref{king}), we have
\begin{align}
\nonumber\frac{Q_n(z)}{p_n(z;\mu)}&=
\frac{\int_{\bbC}\frac{\overline{p_n(w;\mu)}Q_n(w)}{z-w}d\mu(w)}
{\int_{\bbC}\frac{|p_n(w;\mu)|^2}{z-w}d\mu(w)}\\
&\label{split}=\frac{\int_{\bbC}\frac{|p_n(w;\mu)|^2}{z-w}d\mu(w)+\int_{\bbC}\frac{\overline{p_n(w;\mu)}(Q_n(w)-p_n(w;\mu))}{z-w}d\mu(w)}{\int_{\bbC}\frac{|p_n(w;\mu)|^2}{z-w}d\mu(w)}.
\end{align}
By Lemma \ref{below}, the denominator and the matching term in the numerator in (\ref{split}) stay away from $0$, so we need only show the second term in the numerator goes to $0$ as $\nri$.  For this, we apply the Schwarz inequality to see that
\begin{align}\label{beteq}
\left|{\int_{\bbC}\frac{\overline{p_n(w;\mu)}(Q_n(w)-p_n(w;\mu))}{z-w}d\mu(w)}\right|^2\leq \frac{\|Q_n(w)-p_n(w;\mu)\|^2_{L^2(\mu)}}{\inf_{w\in\ch(\mu)}|z-w|^2}.
\end{align}
The norm can be expanded as
\[
\|p_n(\cdot;\mu)\|^2_{L^2(\mu)}+\|Q_n\|^2_{L^2(\mu)}
-2\Real\left[\langle Q_n(w),p_n(w;\mu)\rangle_{\mu}\right].
\]
Our first hypothesis on $Q_n$ implies that the sum of the first two terms tends to $2$ as $\nri$.  By the orthogonality relation, we may replace $Q_n(w)$ in the inner product by $\tau_n\kappa_n^{-1}p_n(w;\mu)$.  We now apply the second hypothesis on $Q_n$ and arrive at (\ref{keyconc}).

To prove the statement concerning uniformity, notice that Lemma \ref{below} proves that convergence holds uniformly on compact subsets of $\bbC\setminus\ch(\mu)$ so by the maximum modulus principle, we get uniformity on any closed set in $\barc\setminus\ch(\mu)$, even those that include infinity.
\end{proof}

In the remaining sections, we will see how one can apply Theorem \ref{nis}.

\section{Application:  Measures Supported on the Unit Disk}\label{dsk}

Now we will present some applications of Theorem \ref{nis} to measures supported on the closed unit disk.  We will pay special attention to ratio asymptotics of the orthonormal polynomials.  Ratio asymptotics for orthonormal polynomials on the unit circle or an interval have been studied extensively (see for example \cite{DenType,RatoArc,Duran,NevaiOP,Rakh1,Rakh2,OPUC1,OPUC2,ToRat}).  We will focus on regular measures on the closed unit disk $\bard$ and restrict ourselves to finding a subsequence along which we have the desired behavior.  In Section \ref{wkmsr}, we will examine the behavior of the measures $\{|p_n(z;\mu)|^2d\mu(z)\}_{n\geq0}$ when $\mu$ is regular on $\bard$.

Before we proceed with the statement and proof of our results, we state the following technical lemma.  We recall that for a set of natural numbers $\mcn\subseteq\bbN$, we define its \textit{asymptotic density} as
\[
\lim_{\nri}\frac{\#\{\mcn\cap\{1,2,\ldots,n\}\}}{n}
\]
provided this limit exists.

\begin{lemma}\label{anotherone}
Let $\mcn\subseteq\bbN$ be a subsequence with asymptotic density $1$.  There exists a subsequence $\mcn_1\subseteq\mcn$ also of asymptotic density $1$ so that if $\ell\in\bbZ$ is fixed then every sufficiently large $m\in\mcn_1$ can be written as $q+\ell$ for some $q\in\mcn$.
\end{lemma}

\noindent\textit{Remark.}  An equivalent condition on $\mcn_1$ in the statement of the lemma is that if $\ell\in\bbZ$ is fixed then for all sufficiently large $m\in\mcn_1$, the set $\{m-|\ell|,m-|\ell|+1,\ldots,m+|\ell|\}$ is contained in $\mcn$.

\begin{proof}
The idea is to think of the set $\bbN\setminus\mcn$ as being gaps in the set $\mcn$ and then to widen the gaps in smart way.  More precisely, let $\mcm=\bbN\setminus\mcn$ and let $[n]=\{1,2,\ldots,n\}$.  If $k\in\bbN$ is fixed, then by definition of asymptotic density, one has
\[
\lim_{\nri}\frac{k|\mcm\cap[n]|}{n}=0,
\]
where $|X|$ denotes the cardinality of the set $X$.  Therefore, by a standard argument, we can find a sequence of natural numbers $\{k_n\}_{n=1}^{\infty}$, which is non-decreasing and is unbounded so that
\[
\lim_{\nri}\frac{k_n|\mcm\cap[n]|}{n}=0
\]

For every $m\in\mcm$, let $U_m=\{m-(k_m-1),\ldots,m,\ldots,m+k_{m}-1\}$ and define
\[
\widetilde{\mcm}=\bigcup_{m\in\mcm}U_m.
\]
Then
\[
\limsup_{\nri}\frac{|\widetilde{\mcm}\cap[n]|}{n}\leq\limsup_{\nri}\frac{2k_n|\mcm\cap[n]|}{n}=0,
\]
so $\bbN\setminus\widetilde{\mcm}$ has density $1$.  Define $\mcn_1=\bbN\setminus\widetilde{\mcm}$and let $\ell\in\bbZ$ be fixed.  Clearly $\widetilde{\mcm}$ is divided into blocks so that the first and last $|\ell|$ elements of any sufficiently large block are not in $\mcm$.  In other words, if we shift every block of $\mcn_1$ to the left or right by $|\ell|$, all but finitely many blocks land in $\mcn$, which is the desired conclusion.
\end{proof}

Lemma \ref{anotherone} easily allows us to establish two relevant results for measures supported on the real line.  The first of these involves ratio asymptotics and can be stated as follows:

\begin{theorem}\label{realc}
Let $\mu$ be a measure supported on some compact subset of the real line.  Assume further that $\mu$ is regular and has essential support equal to $[-2,2]$.  Define $\varphi(w)=\frac{1}{2}(w+\sqrt{w^2-4})$, which is the conformal map from the complement of $[-2,2]$ to the exterior of the closed unit disk.  There exists a subsequence $\mcn\subseteq\bbN$ of asymptotic density $1$ so that
\begin{align}\label{rlrt}
\lim_{{\nri}\atop{n\in\mcn}}\frac{\varphi(z)p_{n-1}(z;\mu)}{p_n(z;\mu)}=1
\end{align}
for all $z\not\in\supp(\mu)$.
\end{theorem}

The second result concerns the weak limits of the measures $\{|p_n(z;\mu)|^2d\mu(z)\}_{n\in\bbN}$.  It can be stated as follows:

\begin{theorem}\label{realiff}
Let $\mu$ be a measure supported on some compact subset of the real line.  Assume further that $\mu$ is regular and has essential support equal to $[-2,2]$.  There exists a subsequence $\mcn\subseteq\bbN$ of asymptotic density $1$ so that
\begin{align}\label{realweak}
\textrm{w-}\lim_{{\nri}\atop{n\in\mcn}}|p_n(x;\mu)|^2d\mu(x)=d\omega(x)
\end{align}
where $\omega$ is the equilibrium measure for the interval $[-2,2]$.
\end{theorem}

We will prove both results simultaneously.

\vspace{2mm}

\noindent\textit{Proofs.}
Let $\{a_n,b_{n}\}_{n\in\bbN}$ be the recursion coefficients for the orthonormal polynomials and the measure $\mu$, that is
\[
xp_n(x;\mu)=a_{n+1}p_{n+1}(x;\mu)+b_{n+1}p_n(x;\mu)+a_np_{n-1}(x;\mu).
\]
Since $\mu$ is regular, then by Theorem 1.1 in \cite{CNClass} and Lemma \ref{anotherone} above, we may find a subsequence $\mcn\subseteq\bbN$ of asymptotic density $1$ so that for every $m\in\bbZ$, we have
\begin{align}\label{buffer}
\lim_{{\nri}\atop{n\in\mcn}}a_{n+m}=1\qquad,\qquad\lim_{{\nri}\atop{n\in\mcn}}b_{n+m}=0.
\end{align}
Theorem \ref{realc} now follows by mimicking the second proof of Theorem 2.1 in \cite{RealRat}.  Similarly, Theorem \ref{realiff} follows by mimicking proof of Proposition 3.3 in \cite{RealRat}.
\begin{flushright}
$\Box$
\end{flushright}

\vspace{2mm}

\noindent\textit{Remark.}  An inspection of the proof of Proposition 3.3 in \cite{RealRat} and the second proof of Theorem 2.1 in \cite{RealRat} reveals that we do not actually need regularity to prove Theorems \ref{realc} and \ref{realiff}.  We only require boundedness of the recursion coefficients and (\ref{buffer}).  By choosing the coefficients $b_n$ to be identically zero and the coefficients $a_n$ to be very small on a sufficiently sparse subsequence, one can construct examples to show that the converse to both results is false. 

\vspace{2mm}

As indicated by (\ref{buffer}), the proofs of Theorems \ref{realc} and \ref{realiff} depend heavily on the existence of a recursion relation satisfied by the orthonormal polynomials.  Our main goal in this section is to prove analogs of (\ref{rlrt}) and (\ref{realweak}) for measures on the closed unit disk; a setting in which the orthonormal polynomials do not in general satisfy a finite term recursion relation.

\subsection{Ratio Asymptotics on the Disk.}\label{diskrat}

We begin with a result that is related to the conjecture in \cite{SaffConj}.  There, it is conjectured that for a measure $\mu$ of a certain form on $\bard$, one has $p_n(z;\mu)/(zp_{n-1}(z;\mu))\rightarrow1$ for all $z$ in $\barc\setminus\bard$.  As a corollary, one then concludes that $\kappa_n\kappa_{n-1}^{-1}\rightarrow1$ as $\nri$ (recall $\kappa_n$ is the leading coefficient of $p_n(\cdot;\mu)$).  We will show that in fact the corollary implies the conjecture.  More precisely, we will show that we need only verify the ratio asymptotic behavior at infinity to deduce it for all of $\barc\setminus\bard$.  This can be viewed as a unit disk analog of Theorem 1.7.4 in \cite{OPUC1}.

\begin{theorem}\label{Saffback}
Let $\mu$ be a measure on $\bard$ and $\mcn\subseteq\bbN$ a subsequence so that
\begin{align}\label{onerat}
\lim_{{\nri}\atop{n\in\mcn}}\kappa_{n}\kappa_{n-1}^{-1}=1.
\end{align}
Then
\[
\lim_{{\nri}\atop{n\in\mcn}}\frac{zp_{n-1}(z;\mu)}{p_n(z;\mu)}=1
\]
uniformly on compact subsets of $\barc\setminus\bard$.
\end{theorem}

\noindent\textit{Remark.}  The condition (\ref{onerat}) does \textit{not} imply $\partial\bbD\subseteq\supp(\mu)$.  Indeed there are examples of measures whose essential support is exactly two points and (\ref{onerat}) holds with $\mcn=2\bbN+1$ (see Example 1.6.14 in \cite{OPUC1}).

\begin{proof}
We will apply Theorem \ref{nis} with $Q_n=zp_{n-1}(z;\mu)$.  We need only verify the first condition in Theorem \ref{nis}; the other condition is immediate from our hypotheses.  The upper bound
\[
\limsup_{\nri,n\in\mcn}\|Q_n\|_{\ltm}\leq1
\]
is obvious while the lower bound follows from Remark 2 following Theorem \ref{nis}.
\end{proof}

From Theorem \ref{Saffback}, we deduce the following corollary, which is an analog of (\ref{rlrt}) for regular measures on the unit disk.  It also tells us that if the conjecture in \cite{SaffConj} is false, then it can only fail along a sparse subsequence.

\begin{corollary}\label{reg}
Let $\mu$ be a regular measure on $\bard$.  There exists a subsequence $\mcn\subseteq\bbN$ of asymptotic density $1$ so that
\begin{align}\label{regurat}
\lim_{{\nri}\atop{n\in\mcn}}\frac{zp_{n-1}(z;\mu)}{p_n(z;\mu)}=1
\end{align}
uniformly on compact subsets of $\barc\setminus\bard$.
\end{corollary}

\noindent\textit{Remark 1.}  We will generalize this result in the example in Section \ref{conn}.

\vspace{2mm}

\noindent\textit{Remark 2.}  In Proposition 3.4 in \cite{SaffConj}, the author verifies boundedness of the ratio (\ref{regurat}) under related hypotheses.

\begin{proof}
To apply Theorem \ref{Saffback}, we need to verify that $\kappa_{n}\kappa_{n-1}^{-1}\rightarrow1$ along some subsequence of asymptotic density $1$.  If we define $\gamma_n=\kappa_{n}\kappa_{n-1}^{-1}$ then each $\gamma_n\geq1$.  Regularity implies $\left(\prod_{j=1}^n\gamma_j\right)^{1/n}\rightarrow1$ so $\gamma_n$ tends to $1$ along a subsequence of asymptotic density $1$ as desired.
\end{proof}

Corollary \ref{reg} cannot be improved to give us convergence as $n$ tends to infinity through all of $\bbN$ as the following example shows.

\vspace{2mm}

\noindent\textbf{Example.}  Let $\mu$ be a probability measure supported on the unit circle.  To every such measure, one can canonically assign a sequence $\{\alpha_n\}_{n=0}^{\infty}$ of complex numbers in the unit disk by setting $\alpha_n=-\overline{\Phi_{n+1}(0;\mu)}$.  Conversely, any such sequence determines a probability measure $\mu$ on $\partial\bbD$ (see Chapter 1 in \cite{OPUC1}).  This sequence satisfies
\begin{align}\label{alph}
\frac{\|\Phi_{n+1}(\cdot;\mu)\|_{\ltm}^2}{\|\Phi_n(\cdot;\mu)\|_{\ltm}^2}=1-|\alpha_n|^2
\end{align}
(see formula (1.5.12) in \cite{OPUC1}).  Let us define the measure $\mu$ by defining
\[
\alpha_n=
\begin{cases}
\frac{1}{2}, & \mbox{if } n=2^j \, \mbox{ for some } j\in\bbN \\
0, & \mbox{otherwise.}
\end{cases}
\]
One can easily see that this measure is regular.   However
\[
\frac{zp_{2^j}(z;\mu)}{p_{2^j+1}(z;\mu)}\bigg|_{z=\infty}=\frac{\sqrt{3}}{2},
\]
so we can only apply Corollary \ref{reg} to the subsequence $\mcn=\bbN\setminus\{2^j+1:j\in\bbN\}$.

\vspace{2mm}

Now let us turn our attention to measures supported on the unit circle $\partial\bbD$.  In this case, the polynomials do satisfy a recursion relation and therefore one can actually strengthen the conclusion of Theorem \ref{nis}.

\begin{theorem}\label{circase}
Let $\mu$ be a probability measure supported on the unit circle and let $Q_n$ be as in Theorem \ref{nis}.  Then
\[
\frac{Q_n(\cdot)}{p_n(\cdot;\mu)}\rightarrow1
\]
in $L^2(\partial\bbD,\dtpi)$.
\end{theorem}

\begin{proof}
We use the Bernstein-Szeg\H{o} Approximation Theorem (Theorem 1.7.8 in \cite{OPUC1}) to calculate
\begin{align*}
\int_0^{2\pi}\left|\frac{Q_n(\eitheta)}{p_n(\eitheta;\mu)}\right|^2\dtpi&=
\int_{\partial\bbD}\left|Q_n(z)\right|^2d\mu(z)\rightarrow1
\end{align*}
by hypothesis.  Theorem \ref{nis} establishes uniform convergence on compact subsets of $\barc\setminus\bard$ and we have just established convergence of norms so the result follows by Theorem 1 in \cite{Keldysh}.
\end{proof}

\subsection{Weak Asymptotic Measures}\label{wkmsr}

Consider now the unit disk analog of (\ref{realweak}).  Lemma \ref{anotherone} tells us that if $\mu$ is a regular measure on $\bard$, then there exists a subsequence $\mcm\subseteq\bbN$ of asymptotic density $1$ so that for every $m\in\bbZ$ we have
\[
\lim_{{\nri}\atop{n\in\mcm}}\kappa_{n+m}\kappa_n^{-1}=1.
\]
To see this, we let $\mcn$ be the subsequence as in Corollary \ref{reg} and let $\mcm$ be the subsequence of $\mcn$ constructed by Lemma \ref{anotherone}.  Then if $m>0$, we have
\[
\frac{\kappa_{n+m}}{\kappa_n}=\frac{\kappa_{n+m}}{\kappa_{n+m-1}}\cdot\frac{\kappa_{n+m-1}}{\kappa_{n+m-2}}\cdot\cdots\cdot\frac{\kappa_{n+1}}{\kappa_n}.
\]
Since $\{n,n+1,\ldots,n+m\}\subseteq\mcn$ whenever $n\in\mcm$ (for large $n$), we see that all of the ratios in the above equality tend to $1$ as $\nri$ through $\mcm$.  A similar argument works if $m<0$.

This observation will allow us to make further conclusions about regular measures supported on $\bard$.  More specifically, we will address possible weak limits of the sequence of probability measures $\{|p_n(z;\mu)|^2d\mu(z)\}_{n\in\bbN}$.  Without the regularity hypothesis, the set of weak limit points can be hard to control.  Indeed, example 8.2.9 in \cite{OPUC1} shows that for measures on $\partial\bbD$, the set of weak limit points of the sequence $\{|p_n(z;\mu)|^2d\mu(z)\}_{n\in\bbN}$ can be all probability measures on $\partial\bbD$.  Theorem 9.3.1 in \cite{OPUC2} tells us that if $\mu$ is supported on $\partial\bbD$ then $|p_n(z;\mu)|^2d\mu(z)\rightarrow\dtpi$ weakly if and only if for every $k\in\bbN$ fixed we have $\Phi_n(0;\mu)\Phi_{n+k}(0;\mu)\rightarrow0$  as $\nri$ (see also Theorem 9.7.3 in \cite{OPUC2}).  It is easy to see that this condition is independent of regularity.

The author and Simon have separate proofs that if $\mu$ is regular on $\partial\bbD$ then
\[
\frac{1}{n+1}\sum_{j=0}^n|p_j(z;\mu)|^2d\mu(z)\rightarrow\dtpi
\]
weakly as $\nri$ (see \cite{SimWeak,WeakCD}).  This suggests convergence along a sequence of density $1$ and we will show this is the case.  In fact, we will show that if $\mu$ is any regular measure on $\bard$ then there is a subsequence $\mcn\subseteq\bbN$ of asymptotic density $1$ so that
\[
w\textrm{-}\lim_{{\nri}\atop{n\in\mcn}}|p_n(z;\mu)|^2d\mu(z)=\dtpi.
\]
The first step is to show that the weak limits we are interested in are measures on $\partial\bbD$.  This is the content of the following lemma.

\begin{lemma}\label{nocompact}
Let $\mu$ be a measure on $\bard$, $m\in\bbN$ fixed, and $\mcn\subseteq\bbN$ a subsequence so that
\[
\lim_{{\nri}\atop{n\in\mcn}}\kappa_{n+m}\kappa_{n}^{-1}=1.
\]
If $K\subseteq\bbD$ is a compact set then
\[
\lim_{{\nri}\atop{n\in\mcn}}\int_K|p_{n}(z;\mu)|^2d\mu(z)=0.
\]
\end{lemma}

\begin{proof}
Let $K\subseteq\bbD$ be a fixed compact set and assume $K\subseteq\{z:|z|<R<1\}$.  For contradiction, let us suppose that there is a subsequence $\mcn_1\subseteq\mcn$ and $\beta>0$ such that
\[
\int_{K}|\Phi_{n}(z;\mu)|^2d\mu\geq\beta\|\Phi_{n}(\mu)\|^2_{L^2(\mu)}
\]
for all $n\in\mcn_1$.
Then for these $n$, we have
\[
\int_{\bard\setminus K}|\Phi_{n}(z;\mu)|^2d\mu\leq(1-\beta)\|\Phi_{n}(\mu)\|^2_{L^2(\mu)}.
\]
We then use the extremal property to calculate
\begin{align*}
\|\Phi_{n+m}(\mu)\|^2_{L^2(\mu)}&\leq\int_K|z^m\Phi_{n}(z;\mu)|^2d\mu+\int_{\bard\setminus K}|z^m\Phi_{n}(z;\mu)|^2d\mu\\
&\leq R^{2m}\int_{K}|\Phi_{n}(z;\mu)|^2d\mu+\int_{\bard\setminus K}|\Phi_{n}(z;\mu)|^2d\mu\\
&= R^{2m}\int_{K}|\Phi_{n}(z;\mu)|^2d\mu+R^{2m}\int_{\bard\setminus K}|\Phi_{n}(z;\mu)|^2d\mu\\
&\qquad\qquad\qquad+(1-R^{2m})\int_{\bard\setminus K}|\Phi_{n}(z;\mu)|^2d\mu\\
&\leq R^{2m}\|\Phi_{n}(\mu)\|^2_{L^2(\mu)}+(1-R^{2m})(1-\beta)\|\Phi_{n}(\mu)\|^2_{L^2(\mu)}\\
&=\left(1-\beta(1-R^{2m})\right)\|\Phi_{n}(\mu)\|^2_{L^2(\mu)},
\end{align*}
which contradicts our hypothesis when $n\in\mcn$ is sufficiently large.
\end{proof}

Now we can prove an analog of (\ref{realweak}) for regular measures on the closed unit disk.

\begin{theorem}\label{regequi}
Let $\mu$ be a regular measure on $\bard$.  There is a subsequence $\mcn\subseteq\bbN$ of asymptotic density $1$ so that
\[
w\textrm{-}\lim_{{\nri}\atop{n\in\mcn}}|p_n(z;\mu)|^2d\mu(z)=\dtpi.
\]
\end{theorem}

\begin{proof}
As mentioned above, we may begin with a subsequence $\mcn\subseteq\bbN$ of asymptotic density $1$ so that for every $m\in\bbN$ we have
\[
\lim_{{\nri}\atop{n\in\mcn}}\kappa_{n+m}\kappa_n^{-1}=1.
\]
It then follows from Lemma \ref{nocompact} that if $K\subseteq\bbD$ is compact, we have
\begin{align*}
\lim_{{\nri}\atop{n\in\mcn}}\int_K|p_n(z;\mu)|^2d\mu(z)=0.
\end{align*}
We conclude that any weak limit of the measures $\{|p_n(z;\mu)|^2d\mu(z)\}_{n\in\mcn}$ is supported on $\partial\bbD$.

Let $\sigma$ be such a weak limit point and $\mcn_{\sigma}\subseteq\mcn$ the corresponding subsequence.  Then for every fixed $k\in\bbN$ we have (by the extremal property)
\begin{align}\label{lown}
\kappa_{n}^2\kappa_{n+k}^{-2}\leq\int_{\bard}|\Phi_k(z;\sigma)p_n(z;\mu)|^2d\mu(z).
\end{align}
As $\nri$ through $\mcn_{\sigma}$, the left hand side of (\ref{lown}) tends to $1$ while the right hand side tends to $\|\Phi_k(\cdot;\sigma)\|^2_{L^2(\sigma)}$.  However, clearly $\|\Phi_k(\cdot;\sigma)\|^2_{L^2(\sigma)}\leq\|z^k\|^2_{L^2(\sigma)}=1$, so we must have $\|\Phi_k(\cdot;\sigma)\|^2_{L^2(\sigma)}=1$, which implies (using notation from the earlier example)
\[
\alpha_j(\sigma)=0,\qquad j=0,1,2,\ldots,k-1.
\]
Since $k\in\bbN$ was arbitrary, this implies $\sigma$ is normalized arc-length measure on $\partial\bbD$ as desired.
\end{proof}

Finally, we conclude this section by exploring the behavior of the ratio (\ref{king}) when $z$ is inside the convex hull of the support of $\mu$ but outside the polynomial convex hull of the support of $\mu$.  The calculations in the proof of Theorem \ref{nis} imply that the second term in the numerator on the right hand side of (\ref{split}) still tends to $0$ in this case, so we can obtain the same conclusion as Theorem \ref{nis} (without the uniformity) if we can show that the denominator on the right hand side of (\ref{split}) stays away from zero (perhaps on some subsequence).

It is clear that if $z\not\in\supp(\mu)$ then the sequence
\[
\left\{\int_{\bbC}\frac{|p_n(w;\mu)|^2}{z-w}d\mu(w)\right\}_{n\geq0}
\]
is bounded uniformly on compact subsets of $\barc\setminus\supp(\mu)$, so Montel's Theorem implies that some subsequence converges uniformly on compact subsets to an analytic function $h(z)$.  It is possible that the limiting function $h(z)$ vanishes at a point inside the convex hull of the support of the measure.  For example, let $\mu$ be a measure supported on $[-2,-1]\cup[1,2]$ satisfying $\mu(A)=\mu(-A)$ for all measurable sets $A$.  Since the measure is symmetric about zero, so are the orthonormal polynomials so we conclude
\[
\int_{\supp(\mu)}\frac{|p_n(w;\mu)|^2}{w}d\mu(w)=0,
\]
i.e. the limiting function $h(z)$ satisfies $h(0)=0$.

However, this example tells us how we can look for the zeros of $h(z)$.  Indeed, Proposition 2.3 in \cite{WeakCD} tells us that if $\mu$ is a regular measure, then for any function $f$ that is analytic in a neighborhood of $\pch(\mu)$ we have
\[
\lim_{\nri}\frac{1}{n+1}\sum_{j=0}^n\int_{\bbC}f(w)|p_j(w;\mu)|^2d\mu(w)=\int_{\bbC}f(w)d\omega_{\mu}(w)
\]
where $\omega_{\mu}$ is the equilibrium measure for the support of $\mu$ (assuming $\mbox{capacity}(\supp(\mu))>0$; see Theorem 3.6.1 in \cite{StaTo}).  Therefore, if
\begin{align}\label{eqzer}
\int_{\bbC}\frac{1}{z-w}d\omega_{\mu}(w)\neq0
\end{align}
then we can find a subsequence $\mcn_z\subseteq\bbN$ of positive density such that
\[
\inf_{n\in\mcn_z}\left\{\left|\int_{\bbC}\frac{|p_n(w;\mu)|^2}{z-w}d\mu(w)\right|\right\}>0.
\]
We conclude that if $\mu$ is regular, $\mbox{capacity}(\supp(\mu))>0$, the hypotheses of Theorem \ref{nis} are satisfied, $z\not\in\pch(\mu)$, and (\ref{eqzer}) holds, then the conclusion (\ref{keyconc}) holds as $n$ tends to infinity through $\mcn_z$.  As mentioned earlier, we will show later by means of an example that one cannot in general extend Theorem \ref{nis} to include the boundary of $\pch(\mu)$ (see Section \ref{chr} below).

\section{Application:  Measures Supported on Regions}\label{conn}

If a measure $\mu$ is supported on an arbitrary region $G$, we cannot prove a result quite as precise as Theorem \ref{realc} or Corollary \ref{reg} using our methods.  The main difficulty is that the conformal maps sending the exterior of $\bard$ to the exterior of $\bard$ or the complement of $[-2,2]$ have finite Laurent expansions, which simplifies matters computationally.  To make up for this, we will approximate the exterior conformal map with polynomials.  The price we will pay is that we will reach a conclusion about $p_n/p_{n-k_n}$ for a possibly unbounded sequence $\{k_n\}$ (but see the example below).

Our proof in this setting will require use of a specific sequence of polynomials called the \textit{Faber polynomials} (see \cite{MDFaber}), which we will denote by $\{F_n(z)\}_{n\geq0}$.  Given a bounded region $G\subseteq\bbC$, let $\Omega$ be the unbounded component of $\barc\setminus\barg$, which is simply connected in the extended complex plane.  Let $\varphi$ denote the conformal map sending $\Omega$ to $\barc\setminus\bard$ satisfying $\varphi(\infty)=\infty$ and $\varphi'(\infty)>0$.  There are three conditions given in \cite{Geronimus} that guarantee the uniform convergence of $F_n-\varphi^n$ to $0$ on $\overline{\Omega}$ as $\nri$.  Whenever this convergence property holds (for example if $G$ satisfies any of the three conditions in \cite{Geronimus}), we will say $G$ is of class $\Gamma$ and write $G\in\Gamma$.  Also note that if $G$ has logarithmic capacity $1$ then $F_n$ is a monic polynomial of degree $n$ for every $n\in\bbN$.  

Our result is the following:

\begin{theorem}\label{keps}
Let $\mu$ be a measure on the closure of a bounded region $G\in\Gamma$ with logarithmic capacity $1$.  Let $\mcn,\mcm\subseteq\bbN$ be infinite subsequences so that for each $j\in\mcm$, $\kappa_{n}\kappa_{n-j}^{-1}\rightarrow1$ as $\nri$ through $\mcn$.  Then there exists a non-decreasing and unbounded sequence $\{k_{n}\}_{n\in\mcn}$ of elements of $\mcm$ such that
\begin{align}\label{gcase}
\lim_{{\nri}\atop{n\in\mcn}}\frac{\varphi^{k_n}(z)p_{n-k_n}(z;\mu)}{p_{n}(z;\mu)}=1
\end{align}
for all $z\not\in\ch(\mu)$.  Furthermore, the convergence is uniform on compact subsets of $\barc\setminus\ch(\mu)$.
\end{theorem}

\begin{proof}
We will apply Theorem \ref{nis} with $Q_n(z)=F_{k_n}(z)p_{n-k_n}(z;\mu)$ for some appropriate $k_n\in\mcm$.  First note that our hypotheses imply that if the sequence $\{k_n\}_{n\in\mcn}$ grows slowly enough then $\kappa_{n}\kappa_{n-k_n}^{-1}$ tends to $1$ as $\nri$ through $\mcn$.  Therefore, the second condition of Theorem \ref{nis} is satisfied by $Q_n$.  Remark 2 following Theorem \ref{nis} puts a lower bound on the $\liminf$ of the $\ltm$-norm of $Q_n$.  To put an upper bound on the $\limsup$, we see
\begin{align*}
\int_{\barg}|F_{k_n}(z)p_{n-{k_n}}(z;\mu)|^2d\mu(z)&\leq\|F_{k_{n}}\|^2_{L^{\infty}(\barg)}
\end{align*}
for every $n\in\bbN$.  Therefore $\|Q_n\|_{\ltm}\leq1+\epsilon_n$ where $\epsilon_n\geq0$ tends to $0$ as $\nri$ through $\mcn$ provided $\{k_{n}\}_{n\in\mcn}$ is unbounded (this is because $G\in\Gamma$ and $|\varphi(w)|=1$ for all $w\in\partial \Omega$).  By invoking Theorem \ref{nis}, we conclude that
\begin{align}\label{fbound}
\lim_{{\nri}\atop{n\in\mcn}}\frac{{F_{k_n}(z)p_{n-k_n}(z;\mu)}}{p_n(z;\mu)}=1
\end{align}
for all $z\nch(\mu)$ and the convergence is uniform on compact subsets of $\barc\setminus\ch(\mu)$.  Since $F_n-\varphi^n$ tends to $0$ on $\overline{\Omega}$ as $\nri$, (\ref{fbound}) implies (\ref{gcase}).
\end{proof}

Although Theorem \ref{keps} is an analog of Theorem \ref{Saffback} for more general supports, proving an analog of Corollary \ref{reg} or Theorem \ref{realc} is more challenging.  The difficulty lies in the fact that it is possible to have $\|\Phi_n(\cdot;\mu)\|_{\ltm}>\|\Phi_{n-1}(\mu)\|_{\ltm}$ when the support of the measure is not the closed unit disk.  The following example shows that we can strengthen the conclusion of Theorem \ref{keps} to more closely resemble that of Theorem \ref{realc} if some power of the conformal map $\varphi$ is a monic polynomial.

\vspace{2mm}

\noindent\textbf{Example.}  Consider the set $E_m:=\{z:|z^m-1|\leq1\}$ (pictured below for $m=3$).
\begin{figure}[h!]\label{GDrop}
\begin{center}
\includegraphics[scale=.77]{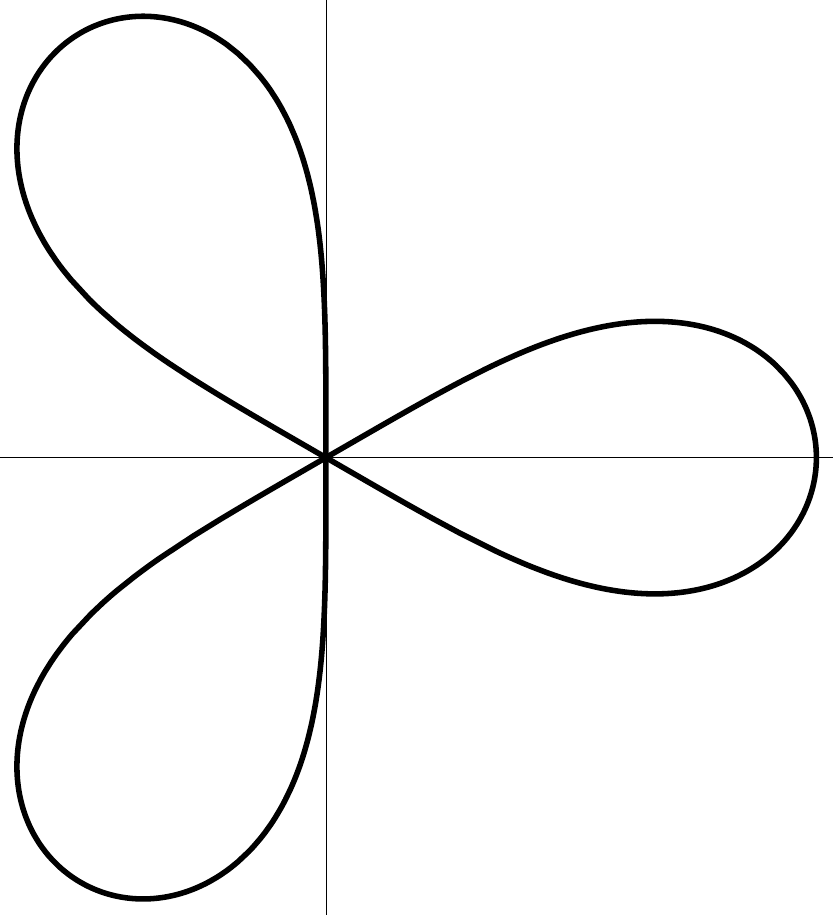}
\caption{The boundary of the set $E_3$.}
\end{center}
\end{figure}
In this case, $F_m(z)=z^m-1$ (see Example 3.8 in \cite{MDFaber}) so that if $\mu$ is a measure supported on $E_m$, we can write $\|\Phi_{n+m}(\cdot;\mu)\|_{\ltm}\leq\|\Phi_n(\cdot;\mu)\|_{\ltm}$ for all $n\in\bbN$.  If $\mu$ is regular, then we have
\[
1=\lim_{\nri}\left(\kappa_n\kappa_{n+1}\cdots\kappa_{n+m-1}\right)^{1/n}=
\lim_{\nri}\left(\kappa_1\cdots\kappa_{m}\prod_{j=1}^{n-1}\kappa_{j+m}\kappa_j^{-1}\right)^{1/n}.
\]
We can now apply the same reasoning as in the proof of Corollary \ref{reg} to conclude that there is a subsequence $\mcn\subseteq\bbN$ of asymptotic density $1$ so that $\lim_{\nri,n\in\mcn}\kappa_n\kappa_{n-m}^{-1}=1$.  Furthermore, $\|F_m(z)\|_{L^{\infty}(E_m)}=1$ so the proof of Theorem \ref{keps} shows that in fact we have
\[
\lim_{{\nri}\atop{n\in\mcn}}\frac{F_m(z)p_{n-m}(z;\mu)}{p_n(z;\mu)}=\lim_{{\nri}\atop{n\in\mcn}}\frac{(z^m-1)p_{n-m}(z;\mu)}{p_n(z;\mu)}=1
\]
for all $z\nch(\mu)$.  Notice that if we set $m=1$ we recover Corollary \ref{reg}.  The same calculation applies in any situation where some power of the conformal map $\varphi$ is a monic polynomial.

\section{Application:  Stability Under Perturbation.}\label{stabl}

\subsection{The Uvarov Transform.}\label{uvar}

Another application of Theorem \ref{nis} is to show that the behavior of the polynomials $\{p_n(z;\mu)\}_{n\geq0}$ is stable under certain perturbations of the measure.  In the following example, we consider the \textit{Uvarov Transform} of a measure (see \cite{Uvarov}), meaning we add a single point mass to the measure $\mu$.

\vspace{2mm}

\noindent\textbf{Example.}
Let $\mu$ be a measure with compact support and $x\in\bbC$.
We will show that for any $t>0$ we have
\begin{align}\label{exone}
\lim_{\nri}\frac{p_{n}(z;\mu+t\delta_{x})}{p_{n}(z;\mu)}=1
\end{align}
uniformly on compact subsets of $\barc\setminus\ch(\mu)$ if and only if
\begin{align}\label{nevailike}
\lim_{\nri}\frac{|p_n(x;\mu)|^2}{K_{n-1}(x,x;\mu)}=0,
\end{align}
where $K_n(y,z;\mu)=\sum_{j=0}^np_j(y;\mu)\overline{p_j(z;\mu)}$.
We will apply Theorem \ref{nis} with $Q_n=p_{n}(z;\mu+t\delta_{x})$.
The proof of Theorem 10.13.3 in \cite{OPUC2} applies in this setting also to show that
\begin{align}\label{normrat}
\frac{\|\Phi_{n}(\cdot;\mu+t\delta_{x})\|^2_{L^2(\mu+t\delta_{x})}}{\|\Phi_{n}(\cdot;\mu)\|^2_{L^2(\mu)}}&=
\frac{1+t K_n(x,x;\mu)}{1+t K_{n-1}(x,x;\mu)}\\
\nonumber&=1+\frac{|p_n(x;\mu)|^2}{K_{n-1}(x,x;\mu)}\cdot\frac{t}{t+K_{n-1}(x,x;\mu)^{-1}}.
\end{align}
Notice that,
\[
\lim_{\nri}\frac{t}{t+K_{n-1}(x,x;\mu)^{-1}}
\]
always exists and lies in the interval $(0,1]$.  Therefore, if we assume (\ref{nevailike}) holds then (\ref{normrat}) verifies the second condition in Theorem \ref{nis} for $Q_n$.  To verify the first condition, write $Q_n=\tau_n\Phi_{n}(\cdot;\mu+t\delta_{x})$ and notice
\[
\|\Phi_{n}(\cdot;\mu+t\delta_{x})\|^2_{L^2(\mu+t\delta_x)}\geq\|\Phi_n(\cdot;\mu)\|^2_{\ltm}+\|\Phi_{n}(\cdot;\mu+t\delta_{x})\|^2_{L^2(t\delta_x)}.
\]
Dividing through by $\|\Phi_{n}(\cdot;\mu+t\delta_{x})\|^2_{L^2(\mu+t\delta_x)}$ and using our above calculations, we get $\|Q_n\|_{L^2(t\delta_x)}\rightarrow0$ as $\nri$, which verifies the first condition in Theorem \ref{nis} and hence proves (\ref{exone}).

If (\ref{nevailike}) does not hold, then (\ref{normrat}) shows that we do not even get the the desired convergence at infinity so we cannot possibly have (\ref{exone}).

\vspace{2mm}

\noindent\textit{Remark 1.}  The condition (\ref{nevailike}) is discussed further in Theorem 10.13.5 in \cite{OPUC2} and also in \cite{BLS}.

\vspace{2mm}

\noindent\textit{Remark 2.}  The Uvarov Transform on the unit circle was studied extensively by Wong in \cite{Wong}.

\vspace{2mm}

In fact, the calculations in the above example prove our next result.  It shows that if a measure is perturbed in a way that does not affect the asymptotic behavior of the monic orthogonal polynomial norms, then it also does not affect the asymptotic behavior of the orthonormal polynomials outside $\ch(\mu)$.

\begin{corollary}\label{normstrng}
Let $\mu_1$ and $\mu_2$ be two measures with compact support such that
\begin{align*}\label{normpres}
\lim_{\nri}\frac{\|\Phi_n(\cdot;\mu_1)\|_{L^2(\mu_1)}}{\|\Phi_n(\cdot;\mu_1+\mu_2)\|_{L^2(\mu_1+\mu_2)}}=1.
\end{align*}
Then
\[
\lim_{\nri}\frac{p_n(z;\mu_1+\mu_2)}{p_n(z;\mu_1)}=1
\]
for all $z\nch(\mu_1)$.
\end{corollary}

\subsection{The Christoffel Transform.}\label{chr}

A second kind of perturbation we will consider is the \textit{Christoffel Transform} of a measure (see \cite{Uvarov}), where we multiply the measure by the square modulus of a monomial; that is, we define
\begin{align}\label{nudef}
d\nu^x(z)=|z-x|^2d\mu(z).
\end{align}
The location of the point $x$ will not be arbitrary; indeed we will have to place a hypothesis on the point $x$ as in (\ref{nevailike}).  We will see later (Corollary \ref{nogrow} below) that this forces $x$ to lie in the convex hull of the support of $\mu$.

For every $n\in\bbN$, we recall the notation $K_n(y,z;\mu)$ to mean the reproducing kernel for polynomials of degree at most $n$ and the measure $\mu$, which is given by
\begin{align}\label{kerdf}
K_n(y,z;\mu)=\sum_{j=0}^np_j(y;\mu)\overline{p_j(z;\mu)}.
\end{align}
A very simple calculation provides us with the following formula (see Proposition $3$ in \cite{Uvarov}):
\begin{align}\label{chrpert}
\Phi_n(z;\nu^x)&=\frac{1}{z-x}\left(\Phi_{n+1}(z;\mu)-\frac{\Phi_{n+1}(x;\mu)}{K_n(x,x;\mu)}K_n(z,x;\mu)\right).
\end{align}

We can now prove the following result:

\begin{theorem}\label{chrstfl}
Let $\mu$ be a measure with compact support and let $\nu^x$ and $\mu$ be related by (\ref{nudef}) where $x$ satisfies (\ref{nevailike}).  Then
\begin{align}\label{addzero}
\lim_{\nri}\frac{(z-x)p_{n-1}(z;\nu^x)}{p_n(z;\mu)}=1
\end{align}
uniformly on compact subsets of $\barc\setminus\ch(\mu)$.
\end{theorem}

\begin{proof}
We wish to apply Theorem \ref{nis} with $Q_n(z)=(z-x)p_{n-1}(z;\nu^x)$.  First notice that
\[
\|Q_n\|^2_{L^2(\mu)}=\frac{\|(\cdot-x)\Phi_{n-1}(\cdot;\nu^x)\|^2_{L^2(\mu)}}{\|\Phi_{n-1}(\cdot;\nu^x)\|^2_{L^2(\nu^x)}}=1
\]
by definition, which verifies the first condition of Theorem \ref{nis}.  By formula (\ref{chrpert}), we calculate
\begin{align*}
\|\Phi_{n-1}(\cdot;\nu^x)\|^{2}_{L^2(\nu^x)}&=\|(\cdot-x)\Phi_{n-1}(\cdot;\nu^x)\|^2_{L^2(\mu)}\\
&=\|\Phi_n(\cdot;\mu)\|^2_{L^2(\mu)}+\frac{|\Phi_{n}(x;\mu)|^2}{K_{n-1}(x,x;\mu)}.
\end{align*}
The leading coefficient $\tau_n$ of $Q_n$ is just $\|\Phi_{n-1}(\cdot;\nu^x)\|^{-1}_{L^2(\nu^x)}$ so we have
\[
\tau_n=\|\Phi_n(\cdot;\mu)\|^{-1}_{L^2(\mu)}(1+o(1))
\]
as $\nri$ by our assumption (\ref{nevailike}).  This verifies the second condition of Theorem \ref{nis} and hence the desired conclusion follows.
\end{proof}

\noindent\textit{Remark.}  By Theorem \ref{circase}, if the measure $\mu$ in Theorem \ref{chrstfl} is supported on the unit circle, then in fact we get $\bbH^2$ convergence in (\ref{addzero}).

\vspace{2mm}

Combining Theorem \ref{chrstfl} with the example in Section \ref{uvar}, we deduce the following corollary:

\begin{corollary}\label{attract}
Let $\mu$ be a measure with compact support, $x\in\bbC$, and $t>0$.  If $x$ satisfies (\ref{nevailike}) then
\[
\lim_{\nri}\frac{(z-x)p_{n-1}(z;\nu^x)}{p_n(z;\mu+t\delta_x)}=1
\]
uniformly on compact subsets of $\barc\setminus\ch(\mu)$.
\end{corollary}

The following example illustrates Theorem \ref{chrstfl} and shows that in general we cannot hope to extend the results of Theorem \ref{nis} to the boundary of $\pch(\mu)$.

\vspace{2mm}

\noindent\textbf{Example.}  Let $\mu$ be two-dimensional area measure on the unit disk $\bbD$ so that $p_n(z;\mu)=\sqrt{\frac{n+1}{\pi}}z^n$.  It is easily seen that in this case, the point $1$ satisfies (\ref{nevailike}) so we will consider the Christoffel Transform given by $\nu^1$.  By the example in Section IV.6 in \cite{Suetin} (or equation (\ref{chrpert}) above), we know that
\[
p_n(z;\nu^1)=\frac{2}{\sqrt{\pi(n+1)(n+2)(n+3)}}\sum_{k=0}^n(k+1)z^k(1+z+z^2+\cdots+z^{n-k}).
\]
We then see that
\begin{align*}
\frac{(z-1)p_n(z;\nu^1)}{p_{n+1}(z;\mu)}&=\\
&\hspace{-24mm}=\frac{2(z-1)}{z^{n+1}(n+2)\sqrt{(n+1)(n+3)}}\sum_{k=0}^n(k+1)z^k(1+z+z^2+\cdots+z^{n-k})\\
&\hspace{-24mm}=\frac{2}{(n+2)\sqrt{(n+1)(n+3)}}\left(\frac{(n+1)(n+2)}{2}-\frac{n+1}{z}-\cdots-\frac{2}{z^n}-\frac{1}{z^{n+1}}\right),
\end{align*}
which clearly tends to $1$ as $\nri$ if $|z|>1$, in accordance with Theorem \ref{chrstfl}.

It is clear that
\[
\frac{(z-1)p_{n-1}(z;\nu^1)}{p_n(z;\mu)}\bigg|_{z=1}=0,
\]
so we cannot in general hope to extend Theorem \ref{nis} to include convergence on the boundary of $\pch(\mu)$.  However, in this example all of the zeros of $p_n(z;\mu)$ are contained in $\bbD$ so $(z-1)p_{n-1}(z;\nu^1)p_n(z;\mu)^{-1}$ is a function in $\bbH^{\infty}(\barc\setminus\bard)$ and as such
\[
\int_{0}^{2\pi}\frac{(\eitheta-1)p_{n-1}(\eitheta;\nu^1)}{p_n(\eitheta;\mu)}\dtpi=\frac{(z-1)p_{n-1}(z;\nu^1)}{p_n(z;\mu)}\bigg|_{z=\infty}=\frac{\kappa_{n-1}(\nu^1)}{\kappa_n(\mu)}\rightarrow1
\]
as $\nri$, which suggests we do have convergence to $1$ almost everywhere on $\partial\bbD$ in this example.  A short calculation reveals that this is the case.

\vspace{2mm}

Theorem \ref{chrstfl} also yields the following (see also Theorem 1.3 in \cite{BLS}):

\begin{corollary}\label{nogrow}
If $x\not\in\ch(\mu)$ then (\ref{nevailike}) fails.
\end{corollary}

\begin{proof}
Since all zeros of $p_n(\cdot;\mu)$ are contained in $\ch(\mu)$, we have
\[
\frac{(z-x)p_{n-1}(z;\nu^x)}{p_n(z;\mu)}\bigg|_{z=x}=0
\]
for every $n\in\bbN$, which means (\ref{nevailike}) cannot possibly hold for otherwise, by Theorem \ref{chrstfl} this expression would have to converge to $1$.
\end{proof}


\vspace{7mm}


\begin{thebibliography}{14}

\bibitem{DenType}  D. Barrios Rolan\'{i}a, B. de la Calle Ysern, and G. L\'{o}pez Lagomasino {\em Ratio and relative asymptotics of polynomials orthogonal with respect to varying Denisov-type measures}, J. Approx. Theory 139 (2006), no. 1-2, 223--256.

\bibitem{RatoArc}  M. Bello Hern\'{a}ndez and G. L\'{o}pez Lagomasino, {\em Ratio and relative asymptotics of polynomials orthogonal on an arc of the unit circle}, J. Approx. Theory 92 (1998), no. 2, 216--244.

\bibitem{Keldysh}  M. Bello-Hern\'{a}ndez, F. Marcell\'{a}n, and J. M\'{i}nguez-Ceniceros, {\em Pseudo-uniform convexity in $H^p$ and some extremal problems on Sobolev spaces}, Complex Variables, 48 (2003), 429--440.

\bibitem{BLS} J. Breuer, Y. Last, and B. Simon, {\em The Nevai condition}, Constr. Approx. 32 (2010), 221--254.

\bibitem{Duran} A. J. Duran, {\em Ratio asymptotics and quadrature formulas}, Constr. Approx. 13 (1997), no. 2, 271--286.

\bibitem{Uvarov}  L. Garza and F. Marcell\'{a}n, {\em Verblunsky Parameters and Linear Spectral Transformations}, Methods and Applications of Analysis 16, (2009), no. 1, 69--86.

\bibitem{Geronimus} Ja. L. Geronimus, {\em Some extremal problems in $L_p(\sigma)$ spaces}, Math Sbornik 31 (1952), 3--23. [In Russian]

\bibitem{MDFaber}  E. Mi\~{n}a-D\'{i}az, {\em On the asymptotic behavior of Faber polynomials for domains with piecewise analytic boundary}, Constr. Approx. 29 (2009), 421--448.

\bibitem{NevaiOP} P. Nevai, {\em Orthogonal polynomials}, Mem. Amer. Math. Soc. 18 (1979), No. 213, 185 pp.

\bibitem{Rakh1}  E. A. Rakhmanov, {\em The asymptotic behavior of the ratio of orthogonal polynomials}, Math. USSR Sb. 32 (1977), 199--213.

\bibitem{Rakh2}  E. A. Rakhmanov, {\em The asymptotic behavior of the ratio of orthogonal polynomials II}, Math. USSR Sb. 46 (1983), 105-117.

\bibitem{Ransford} T. Ransford, {\em Potential Theory in the Complex Plane}, Cambridge University Press, New York, NY, 1995.

\bibitem{SaffConj} E. B. Saff, {\em Remarks on relative asymptotics for general orthogonal polynomials}, Contemp. Math. Journal, 507, Amer. Math. Soc., Providence, RI (2010), 233--239.

\bibitem{SimWeak} B. Simanek, {\em Weak convergence of CD kernels: A new approach on the circle and real line}, J. Approx. Theory 164 (2012), no. 1, 204--209.
    
\bibitem{SimBlob} B. Simanek, {\em Asymptotic properties of extremal polynomials corresponding to measures supported on analytic regions}, submitted.

\bibitem{OPUC1} B. Simon, {\em Orthogonal Polynomials on the Unit Circle, Part One: Classical Theory}, American Mathematical Society, Providence, RI, 2005.

\bibitem{OPUC2} B. Simon, {\em Orthogonal Polynomials on the Unit Circle, Part Two: Spectral Theory},
    American Mathematical Society, Providence, RI, 2005.
    
\bibitem{RealRat} B. Simon, {\em Ratio asymptotics and weak asymptotic measures for orthogonal polynomials on the real line}, J. Approx Theory 126 (2004), 198--217.
    
\bibitem{WeakCD} B. Simon, {\em Weak convergence of CD kernels and applications}, Duke Math. J. 146 (2009), 305--330.

\bibitem{CNClass} B. Simon, {\em Regularity and the Ces\`{a}ro-Nevai class}, J. Approx. Theory 156 (2009), no. 2, 142--153.

\bibitem{StaTo} H. Stahl and V. Totik, {\em General Orthogonal Polynomials}, Cambridge University Press, Cambridge, 1992.

\bibitem{Suetin} P. K. Suetin, {\em Polynomials Orthogonal Over a Region and Bieberbach Polynomials},
American Mathematical Society, Providence, RI, 1974.

\bibitem{ToRat} V. Totik, {\em Orthogonal polynomials with ratio asymptotics}, Proc. Amer. Math. Soc. 114 (1992), no. 2, 491--495.

\bibitem{ToTrans} V. Totik, {\em Christoffel functions on curves and domains}, Trans. Amer. Math. Soc. 362 no. 4 (2010), 2053--2087.

\bibitem{Wid2} H. Widom, {\em Polynomials associated with measures in the complex plane}, J. Math. Mech 16 (1967), 997--1013.

\bibitem{Wong} M. Wong, {\em Asymptotics of orthogonal polynomials and point perturbation on the unit circle}, J. Approx. Theory 162 (2010), no. 6, 1294--1321.




\end{thebibliography}
\end{document}